\documentclass[12pt,a4paper]{amsart}

 \usepackage{amssymb, amstext, amscd, amsmath, amsfonts, amsthm, amscd, color}

\usepackage{tikz,mathdots,enumerate}
\usepackage{pgfplots}
\usepackage{url}

\usepackage{kbordermatrix} 

\usepackage{tikz,mathdots,enumerate}

\usepackage{graphicx, array, blindtext}
\usepackage{url}
\usepackage{tikz}
\usetikzlibrary{shapes.geometric, calc}
\usepackage{caption}
\usepackage{array}
\usepackage{epsfig}
\usepackage{eucal}
\usepackage{latexsym}
\usepackage{mathrsfs}
\usepackage{textcomp}
\usepackage{verbatim}

\newtheorem{thm}{Theorem}[section]

\newtheorem{lemma}[thm]{Lemma}

\newtheorem{prop}[thm]{Proposition}

\numberwithin{equation}{section}

\theoremstyle{definition}
\newtheorem{rem}[thm]{Remark}
\newtheorem{example}[thm]{Example}
\newtheorem{definition}[thm]{Definition}
\newcommand{\bN}{{\mathbb{N}}}

\newcommand{\bR}{{\mathbb{R}}}

  \newcommand{\U}{{\mathcal{U}}}

  \newcommand{\X}{{\mathcal{X}}}
  \newcommand{\Y}{{\mathcal{Y}}}
  
\usepackage{verbatim}

\usepackage{tikz}
\usetikzlibrary{calc}

\begin{document}



\title[Equilibrium stresses and rigidity for infinite tensegrities and frameworks]{Equilibrium stresses and rigidity for infinite tensegrities and frameworks}
\author[S. C. Power]{S. C. Power}

\address{Dept.\ Math.\ Stats.\\ Lancaster University\\
Lancaster LA1 4YF \\U.K. }

\email{s.power@lancaster.ac.uk}

\begin{abstract}  Asymptotic equilibrium stresses are defined for countably infinite tensegrities and generalisations of the Roth-Whiteley characterisation of first-order rigidity are obtained. Generalisations of prestress stability and second order rigidity are given for countably infinite bar-joint frameworks and are shown to give sufficient conditions for continuous rigidity relative to certain prescribed motions. The proofs are based on a new short proof for finite frameworks that prestress stability ensures continuous rigidity.
\end{abstract}

\thanks{
{MSC2020 {\it  Mathematics Subject Classification.}
52C25 \\
}}

\maketitle


A finite tensegrity in $\bR^d$ can be conceived of as a structure consisting of inextensible cables, incompressible struts and rigid bars, all of which are connected at their endpoints in the manner of a linearly embedded graph $G$. It is denoted as $G(p)$, where $p=\{p_1,\dots ,p_n\}$ is the embedding (or placement) of the vertices of $G$ and where it is understood that certain edges of $G$ correspond to cables, struts and bars, the so-called {members} of $G(p)$.
Roth and Whiteley \cite{rot-whi} obtained a useful characterisation of
their first-order rigidity, also known as infinitesimal rigidity (IR). This is given in terms of the existence of a proper equilibrium stress together with an evident necessary condition, namely that the structure with all members replaced by bars should be first-order rigid. See Theorem \ref{t:rothwhiteley}. We consider here countably infinite tensegrities and 
obtain generalisations of the Roth-Whiteley theorem with respect to first-order rigidity relative to asymptotically decaying motions, such as first-order $c_0$-rigidity, for velocity fields that tend to zero at infinity, and  first-order $\ell^2$-rigidity, for finite energy velocity fields. In fact we consider  first-order \emph{$\X$-rigidity}, the first-order rigidity with regard to velocity fields of $\X\otimes \bR^d$, where $\X$ is one of the classical Banach sequence spaces $\ell^q, 1\leq q<\infty$, or $c_0$.

 We also consider equilibrium stresses from the point of view of generalising the notion of prestress stability to countably infinite tensegrities and bar-joint frameworks.
Recall that the {continuous rigidity} (CR) of a finite bar-joint framework 
requires that there are no nontrivial continuous motions $t\to p(t)$, for $ t\in[0,1]$, which preserve the length constraints on members. It is well-known  that continuous rigidity is ensured by prestress stability (PS) and also by second order rigidity (2OR).
The main references here are Connelly   
 \cite{con-1980}, \cite{con-energy} and Connelly and Whiteley \cite{con-whi}.
See also the more recent articles of Connelly and Gortler \cite{con-gor} and Holmes-Cerfon, Gortler and Theran \cite{hol-gor-the}.
We give new short proofs of these implications which use only Connelly's stress energy function.  The arguments also adapt readily to the countably infinite setting and we obtain analogous implications for
our definitions of bounded prestress stability (BPS) and a restricted form of continuous rigidity, namely \emph{directed continuous rigidity} (DCR). See Definition \ref{d:boundedlysmooth} and Remark \ref{r:CRandDCR}. This is of interest since, as shown in Section \ref{ss:examples},
infinitesimal rigidity does not imply continuous rigidity for infinite frameworks.



 In Section \ref{s:section1} we define tensegrities and their infinitesimal (first-order) flexes and give a proof of the Roth-Whiteley theorem. 
In Section 2 we consider countably infinite tensegrity frameworks noting first that the existence of a proper equilibrium stress need not certify first-order rigidity.
On the other hand we show that the existence of a certain \emph{proper asymptotic equilibrium stress} (Definition \ref{d:aestress})  does imply first-order $\X$-rigidity. In Section 3, which is essentially independent of the previous sections, we consider prestress stability for infinite bar-joint frameworks.

The deeper direction of the Roth-Whiteley characterisation is that  first-order rigidity ensures the presence of a proper equilibrium stress. To generalise this, to asymptotic equlibrium stresses and first-order $\X$-rigidity, we consider closed convex cones $C$ in Banach sequence spaces and make use of Hahn-Banach separation functionals and a relative second dual cone equality lemma, namely Lemma \ref{l:doubledualconeSUB}.

Much of the basic theory of first-order rigidity for finite bar-joint frameworks extends in some manner to countably infinite bar-joint frameworks, particularly in the case of generic frameworks. See Kitson and Power  \cite{kit-pow-I}, \cite{kit-pow-II} for example. On the other hand countably infinite tensegrities have received little attention in this regard. 
 E. B. Ashton \cite{ash} has considered ``continuous tensegrities" and certain extensions of the Roth-Whiteley theorem but the main applications involve tensegrities with a continuum of struts, in a continuous path, with ends connected by curved cables. As we have remarked, a literal generalisation of the Roth-Whiteley characterisation to countably infinite tendesgrities does not hold for general first order rigidity. The adjustments we make are natural from the perspective of functional analysis and in particular we restrict attention to countable tensegrities $G(p)$ which have \emph{uniform structure} in the sense that there is an upper bound both to the lengths of members and to the degrees of the vertices of $G$. This ensures that the rigidity matrix determines a bounded linear transformation from $\X\otimes \bR^d$ to $\X\otimes \bR$ and that its transpose matrix is a bounded linear transformation in the reverse direction.  For other considerations of the rigidity matrix as a bounded linear transformation see also Owen and Power \cite{owe-pow-crystal} and Kastis, Kitson and Power \cite {kas-kit-pow}.

The proofs below 
are self-contained apart from standard separation results for closed cones in Hilbert space and Banach spaces. 
The recent book of Connelly and Guest \cite{con-gue} gives a useful reference  to the rigidity theory of finite bar-joint frameworks and tensegrities. 



\section{Finite and infinite cable-strut tensegrities}\label{s:tensegrites}\label{s:section1}

A \emph{tensegrity} $G(p)$ is a bar-joint framework $(G,p)$ in $\bR^d$ for which certain bars $p_ip_j$, for $ij$ in a subset $E_c$ (resp.  $E_s$),  have been replaced by \emph{cables} (resp. \emph{struts}). This arises from a partition $E=E_b\cup E_c\cup E_s$ of the edges of the underlying simple graph $G=(V,E)$. The infinitesimal constraint equations, associated with a finite tensegrity $G(p)$ and the indexing $V=\{v_1,\dots ,v_n\}$, are the conditions for a velocity field $u=\{u_1,\dots ,u_n\} \subset \bR^d$ and placement $p=(p_1,\dots ,p_n)$ which are given by
\begin{eqnarray*}
(p_i-p_j)\cdot(u_i-u_j)&=0, \quad ij \in E_b,\\
(p_i-p_j)\cdot(u_i-u_j)&\leq 0, \quad ij \in E_c,\\
(p_i-p_j)\cdot(u_i-u_j)&\geq 0, \quad ij \in E_s.
\end{eqnarray*}
When these conditions hold $u$ is said to be an \emph{infinitesimal flex} of $G(p)$. If all infinitesimal flexes $G(p)$ are trivial, that is, are rigid motion flexes in the usual sense for bar-joint frameworks, then $G(p)$ is said to be \emph{first-order rigid}. These definitions also apply to countably infinite tensegrities $G(p)$.

It is convenient to write $\overline{G(p)}$ for the associated bar-joint framework $(G,p)$ and we note that
if ${G(p)}$ is first-order rigid then so too is $\overline{G(p)}$ since any infinitesimal flex of the bar-joint framework is also an infinitesimal flex for ${G(p)}$. 
An \emph{equilibrium stress} $\omega:E \to \bR$ for the bar-joint framework  
$\overline{G(p)}$ is a scalar field (or \emph{stress field}) $\omega = (\omega_e)$ 
such that for each joint $p_j$ we have the equilibrium condition
\[
\Sigma_{i:ij \in E} ~~\omega_{ij}(p_i-p_j) =0.
\]
An  \emph{equilibrium stress} $\omega:E \to \bR$ for the tensegrity $G(p)$ is an
equilibrium stress $\omega$ for $\overline{G(p)}$ with the additional property that $\omega_{ij}$ is nonpositive (resp. nonnegative) if $ij$ indexes a cable (resp. strut). Also an equilibrium stress $\omega$ of $G(p)$ is said to be \emph{proper} if $\omega_{ij}$ is nonzero for all cables and struts.

Since a bar constraint for the pair $u_i, u_j$ is equivalent to a cable constraint together with a strut constraint, for the same pair $ij$, it is convenient to consider \emph{cable-strut tensegrities} $G(p)$ determined by $p$ and subsets $E_c, E_s$ which are not necessarily disjoint subsets of  $\{1,\dots ,n\}^2$ (or of $\bN^2$ if $G$ is countable). The constraint system for this cable-strut tensegrity $G(p)$ is then the set of inequalities
\begin{eqnarray*}
(p_i-p_j)\cdot(u_i-u_j)&\leq 0, \quad ij \in E_c,\\
(p_i-p_j)\cdot(u_i-u_j)&\geq 0, \quad ij \in E_s.
\end{eqnarray*}
Note that we retain the fact that $G$ is the associated simple graph, so that a label $ij$ in both $E_c$ and $E_s$ corresponds to one edge of $G$. Also, $\overline{G(p)}$ is the associated bar-joint framework, as before. Henceforth, and without loss of generality, we consider only cable-strut tensegrities $G(p)$.

The inequality constraint system for $G(p)$ can be expressed in terms of a matrix condition for a certain   \emph{tensegrity rigidity matrix} $R(G(p))$. If $G$ is finite then the matrix has $m=|E_c|+|E_s|$ rows and $n$ edges, where  the row for a member $e$ is equal to
\[
\kbordermatrix{& & & & v_i & & & & v_j & & & \\
e & 0 & \cdots &0 & sgn(e)(p_i-p_j) &0& \cdots&0 &sgn(e)(p_j-p_i) &0& \cdots & },
\]
where $sgn(e)$ is $-1$  (resp. $+1$) if $e$ is a cable (resp. strut). It follows that a velocity field $u$ is an infinitesimal flex of $G(p)$ if and only if
the stress field $R(G(p))u$ lies in the cone $\bR^m_+$. Here and below we view
$u$ as a column matrix and $R(G(p))u$ as the usual matrix product.  
Note also that $R(G(p))$ is obtained from the usual rigidity matrix $R(G,p)$ by repeating rows corresponding to cable-strut pairs and then by multiplying the rows corresponding to cables by -1. (The matrix $R(G,p)$ has $|E|$ rows of the form above with all signs positive.) It follows that $R(G(p))u=0$ if and only if $R(G,p)u=0$. Also, $\omega$ is an equilibrium stress of $G(p)$ if and only if $|\omega | R(G(p))u=0$.

 Let $G$ be a countably infinite simple graph with infinite cable-strut tensegrity $G(p)$
with $p=(p_1, p_2, \dots )$ a vertex placement in $\bR^d$. The infinite rigidity matrix $R(G(p))$ can be viewed as a linear transformation from the space of velocity fields $u$ in the direct product vector space $(\bR^d)^\infty$ to the vector space  $\bR^E = \prod_{e\in E}\bR$. We shall assume that each vertex of $G$ has finite degree. In this case the map $\omega \to \omega R(G(p))$ is well-defined as a linear transformation.



The Roth-Whiteley theorem for finite cable-strut tensegrities takes the following form. 

\begin{thm}\label{t:rothwhiteley}
Let $G(p)$ be a finite cable-strut tensegrity. Then $G(p)$ is first-order rigid if and only if $\overline{G(p)}$ is first-order rigid and there exists a proper equilibrium stress.
\end{thm}

The sufficiency direction of the proof is straightforward.
For a stress field $\omega$ with nonnegative values for struts and nonpositive values for cables, let $\mu_e=|\omega_e|$.
Then $\omega$ satisfies the equilibrium equations if and only if $\mu R(G(p))=0$. 
Now let $\overline{G(p)}$ be first-order rigid, let $\omega$ be a proper equilibrium stress for $G(p)$ and let $u$ be an infinitesimal flex of $G(p)$, so that $\sigma= R(G(p))u$ lies in the cone $\bR^m_+$. Then 
\[
\mu \cdot \sigma = \mu \dot (R(G(p)u) = (\mu R(G(p))\cdot u = 0.
\]
Since $\mu_e$ is strictly positive for every member $e$ it follows that $\sigma=0$. Thus $u$ is an infinitesimal flex of $\overline{G(p)}$.  By our assumptions, $u$ is a trivial rigid motion flex and it follows that $G(p)$ is first-order rigid.


The necessity direction depends on some basic properties of convex cones in $\bR^m$ which are brought into play for the convex cone generated by the rows of the rigidity matrix $R(G(p))$.

A {cone} $C$ in  $\bR^m$ is a nonempty set closed under multiplication by nonnegative real numbers. 
The dual cone of $C$ is the cone $C^*$ in the dual space $(\bR^m)'$ consisting of the set of linear functionals $x^*$ such that $x^*(x)\geq0$ for all $x$ in $ C$. A standard fact is that if $C$ is norm-closed then the second dual cone $C^{**}$ coincides with $C$ (that is, with the canonical image of $C$ in the second dual space). This follows readily from the following separation lemma. A proof via standard  Hilbert space geometry also yields a generalisation of the lemma to closed cones in Hilbert space. See also Lemma \ref{l:coneseparation}.

\begin{lemma}\label{l:separationHilbert} {\bf Separation lemma.} Let $C$ be a closed convex cone in $\bR^m$
and $w\notin C$. Then there is a continuous linear functional $f$ with $f(w)< 0 \leq f(x)$ for all $x\in C$.
\end{lemma}



\begin{lemma} {\bf Dual cone lemma.} Let $X=\{x_1,\dots ,x_m\}$ be a set of vectors in $\bR^m$ with convex cone $C(X)=\{\sum_i \lambda_ix_i: \lambda_i\geq 0\}$. Then the following statements are equivalent.
\medskip

(i) $C(X)^*$ is a subspace.

(ii) $C(X)^*= X^\perp$.

(iii)  $\sum_i \mu_ix_i = 0$ for some choice of strictly positive coefficients $\mu_1,\dots ,\mu_m$.
\end{lemma}

\begin{proof} The equivalence of (i) and (ii) is immediate. If (i) holds then   the cone $C(X)=C(X)^{**}$ is also a subspace. Thus for each index $j$ 
\[
-x_j=\sum_i \lambda_{i,j}x_i \quad \mbox{with}\quad \lambda_{i,j}\geq 0,
\]
and so
\[
0=\sum_i \mu_{i,j}x_i \quad \mbox{with}\quad \mu_{i,j}\geq 0, ~~ \mu_{j,j}>0.
\]
Adding these $m$ equalities gives the desired coefficients for (iii).

If (iii) holds then for  $x^*\in C(X)^*$ we have
$
0=x_*(\sum_j \mu_{j}x_j) = \sum_j \mu_{j}x_*(x_j).
$
Since $x_*(x_j)\geq 0$ for each $j$ it follows from the strict positivity of the coefficients that $x_*(x_j)= 0$ for each $j$. Thus $ X^\perp \subseteq C(X)^*$. Since the reverse inclusion is elementary, (ii) holds. 
\end{proof}

\begin{lemma}\label{l:dichotomy} {\bf Dichotomy lemma.}
 Let $A$ be an $m$ by $n$ real matrix. Then precisely one of the following is true.

(i) There exists $u$ in $\bR^m$ with $Au \neq 0$ and $Au \in \bR^m_+$. 

(ii) There exists a row vector $\mu$ with strictly positive entries such that $\mu A=0$.
\end{lemma}

\begin{proof}
Let $X=\{x_1,\dots ,x_m\}$ be the set of row vectors of $A$. 
Then 
\begin{equation*}
\{u: Au=0\}=X^\perp,\quad \quad 
\{u:Au\in \bR^m_+\}=C(X)^*.
\end{equation*}
Note that (i) is false if and only if $C(X)^* = X^\perp$. By the previous lemma this occurs if and only if $\sum \mu_ix_i = 0$ for some choice of strictly positive coefficients and so (i) is false if and only if (ii) is true.
\end{proof}

To complete the proof of Theorem \ref{t:rothwhiteley} let $G(p)$ be first-order rigid. If $u$ is an infinitesimal flex, with $R(G(p))u\in \bR^m_+$ then $u$ must be a rigid motion flex and so 
$R(G(p))u =0$. This means  condition (i) in Lemma \ref{l:dichotomy} does not hold for  $A=R(G(p))$. Therefore (ii) holds and this strictly positive vector $\mu$ determines a proper equilibrium stress $\omega$. Also $\overline{G(p)}$ is first-order rigid since this property holds for $G(p)$.


\section{First-order rigidity for infinite tensegrities}\label{s:infinitetensegrities}

That the Roth-Whiteley equivalence fails to hold for countable tensegrities is evident from the following two examples.

\begin{example}\label{e:RWfails} 
Let $G(p)$ be the infinite periodic tensegrity framework in $\bR^2$ whose placement is the periodic tiling by equilateral triangles and where all the members are cables. Then $\overline{G(p)}$ is first-order rigid and there exist proper equilibrium stresses. Indeed, any stress field which is constant on any  straight line through a cable and takes negative values is such a stress. However, $G(p)$ is not first-order rigid since, for example, the contractive affine map $(x,y)\to (x/2,y/2)$ restricts to a velocity field which is an infinitesimal flex of $G(p)$. On the other hand we shall see that $G(p)$ is first-order $c_0$-rigid.
\end{example}

\begin{example}\label{e:RWfailsB}
A countable cable-strut tensegrity $G(p)$ in $\bR^2$ is indicated in Figure \ref{f:tensegritystrip1stress}. The dashed lines represent  cables between the vertex placements (the joints of $G(p)$) while the solid lines represent cable-strut pairs (equivalent to bars). The nondiagonal members all have length 1. While the companion bar-joint framework $\overline{G(p)}$ is first-order rigid, the tensegrity $G(p)$ has nontrivial infinitesimal flexes in which countably many of the upper joints have the same leftward translational velocity while all other joints are fixed. This leads to the observation that the set of infinitesimal flexes of $G(p)$ is a convex cone which is not finitely generated. The figure also indicates an equilibrium  stress field $\omega$ for $G(p)$, where a ``0" label on a cable-strut pair indicates that the total stress  $\omega_{ij,c}
+\omega_{ij,s}$ is equal to 0 for these 2 members. Choosing nonzero values for these pairs leads to $\omega$ being a proper equilibrium stress for $G(p)$.
A consequence of this example is that there is no literal counterpart to the dichotomy lemma for an infinite matrix $A$, even when $A$ has the sparse structure of a rigidity matrix. 
\begin{center}
\begin{figure}[ht]
\centering
\includegraphics[width=6cm]{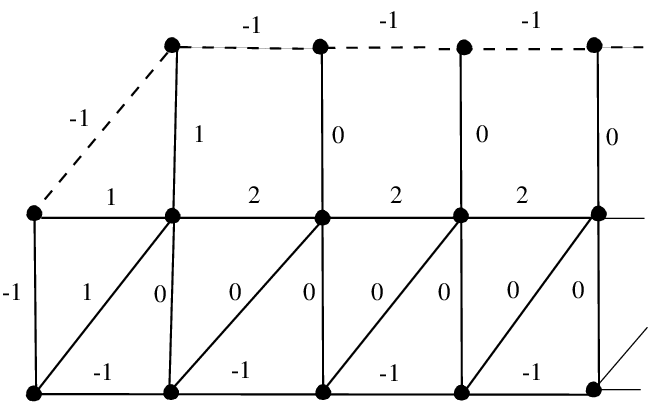}
 \caption{A countable cable-strut tensegrity with a proper equilibrium stress.}
  \label{f:tensegritystrip1stress}
\end{figure}
\end{center}
\end{example}

Let $\X$ be one of the Banach spaces $c_0$ or $\ell^q$, for $1 \leq q <\infty$.  Then the tensor product space $\X \otimes \bR^d$ may be viewed as a space of velocity fields for any countable cable-strut tensegrity $G(p)$. More precisely, this is the vector space of velocity fields $u = (u_1, u_2, \dots )$, with $u_k\in \bR^d$, which is a Banach space for a tensor product norm $\|\cdot\|=\|\cdot\|_\X\otimes \|\cdot\|'$, where $\|\cdot\|'$ is a norm on $\bR^d$. These norms are equivalent norms and in this sense independent of $\|\cdot\|'$. Let us take $\|\cdot\|'=\|\cdot\|_q$ when $\X=\ell^q$ and $\|\cdot\|'=\|\cdot\|_\infty$ when $\X=c_0$. Then the normed vector-valued sequence space
$\X\otimes \bR^d$ is naturally isometrically isomorphic to $\X$.


\begin{definition}\label{d:Xrigidity}
Let $G(p)$ be a countable tensegrity in $\bR^d$ 
and let $\X$ be a classical Banach sequence space. Then $G(p)$ is first-order $\X$-rigid (or $\X$-infinitesimally rigid) if every infinitesimal flex of $G(p)$ in  $\X \otimes \bR^d$ is trivial.
\end{definition}

In particular while the ``strip tensegrity" in Example \ref{e:RWfailsB} is infinitesimally flexible it is nevertheless $c_0$-infinitesimally rigid. To see this note that the subframework determined by the lower two rows of joints is an infinitesimally rigid tensegrity. Let $u$ be a general infinitesimal flex of $G(p)$. By adding a trivial (rigid motion) infinitesimal flex we may assume that the the velocity $u(p)$ at every joint $p$ of this subframework is zero. If $u(p)$ is nonzero for some joint $p$ in the upper row of cables then it is necessarily of the form $(-a,0)$ with $a>0$. It follows that the velocity $u(p')$ for any joint in the top row to the right of $p$ has the form
$(-a',0)$ with $a'\geq a$, and so $u$ is not in $c_0 \otimes \bR^d$.

We now define some proper asymptotic equilibrium stresses for a tensegrity $G(p)$ which are associated with the Banach sequence space $\X$. We write $\X'$ for the dual sequence space of $\X$. Recall that the rigidity matrix $R(G(p))$ accommodates the negative signs (for cables) of an equilibrium stress in the sense that a stress field $\omega$ is an equilibrium stress if and only if $|\omega|R(G(p))=0$.

\begin{definition}\label{d:aestress}
Let $G(p)$ be a countable cable-strut tensegrity and let $\omega^{(n)}$, for $n=1,2,\dots ,$ be a sequence of finitely nonzero stress fields for $G(p)$ such that 
\medskip

(i) $\omega^{(n)}_e\leq 0$ if $e$ is a cable and  $\omega^{(n)}_e\geq 0$ if $e$ is a strut,
\medskip

(ii) for every member $e$ there is $\delta_e>0$ such that $|\omega^{(n)}_e|\geq \delta_e$ for sufficiently large $n$.
\medskip

Then  
\medskip

(a) $\omega^{(n)}$ is a \emph{proper $\X'$-asymptotic equilibrium stress} for $G(p)$ if $|\omega^{(n)}|R(G(p))$  tends to 0 in $\X'\otimes \bR^d$ as $n\to \infty$,
\medskip

(b) $\omega^{(n)}$ is a \emph{proper $(\X',\X)$-asymptotic equilibrium stress} for $G(p)$ if $|\omega^{(n)}|R(G(p))$  tends to 0 in the weak topology of
$\X'\otimes \bR^d$.
\end{definition}

Note that there is no requirement that the sequence $\omega^{(n)}$ is subject to any kind of convergence.
The convergence condition in (b) ensures that $|\omega^{(n)}|R(G(p))u$  tends to 0 when the velocity field $u$ is in $\X\otimes \bR^{d}$, and so for our contexts below (a) is a stronger requirement than (b).

We consider tensegrities 
$G(p)$ with \emph{uniform structure} by which we mean that there is an upper bound for the degrees of the vertices of $G$ and an upper bound for the lengths of the members of $G(p)$. In this case $R(G(p))$ is a bounded linear transformation from $\X \otimes \bR^d$ to $\X \otimes \bR$ (since the bar lengths are uniformly bounded) and the transpose matrix $R(G(p))^T$ is a bounded linear transformation from $\X \otimes \bR$ to $\X \otimes \bR^d$ (since $G(p)$ has uniform structure). 

In the next lemma we note that a proper equilibrium stress in $\X'\otimes \bR$, for $\X\neq \ell^1$, provides a proper $\X$-asymptotic equilibrium stress. 

\begin{lemma}\label{l:ESimpliesasymptoticES}
Let $G(p)$ be a cable-strut tensegrity with uniform structure and let $\X$ be $c_0$ or $\ell^q$, for $1<q<\infty,$ with dual space $\X'$. Let $\omega = (\omega_1, \omega_2, \dots )$ be a proper equilibrium stress field for $G(p)$ with  $\omega \in\X' \otimes \bR$. For each $n$ let $\omega^{(n)}$ be the 
stress field $(\omega_1, \dots , \omega_n, 0, 0, \dots)$. Then the sequence $(\omega^{(n)})$ is a proper $\X'$-asymptotic equilibrium stress for $G(p)$.
\end{lemma}

\begin{proof} In view of the hypotheses for $\X$ the sequence $\omega^{(n)}$ tends to $ \omega$ in $\X\otimes \bR^d$ as $n\to \infty$. Since $G(p)$ has uniform structure,
$|\omega^{(n)}| R(G(p))\to |\omega| R(G(p))=0.$
%
\end{proof}

Any proper equilibrium stress for the strip framework of Example  \ref{e:RWfailsB} 
must have the same negative value $\omega_e$ for all the cable members $e$ and so there is no  proper equilibrium stress in $\X' \otimes \bR$ when $\X$ is $c_0$ or $\ell^q, 1 < q<\infty$. However, there do exist proper $(\X',\X)$-asymptotic equilibrium stresses which are obtained by truncating, in the manner of the previous lemma, the $\ell^\infty \otimes \bR$ equilibrium stress given in Example  \ref{e:RWfailsB}. 

The next lemma  generalises the sufficiency direction of the Roth-Whiteley theorem. 

\begin{lemma}\label{l:generalsufficiency}
Let $G(p)$ be a countable cable-strut tensegrity with uniform structure, let $\X$ be one of the Banach sequence spaces $c_0, \ell^q,$ for $ 1\leq q <\infty$, and suppose that
$\overline{G(p)}$ is first-order $\X$-rigid and there exists a proper $(\X',\X)$-asymptotic equilibrium stress.
Then $G(p)$ is first-order $\X$-rigid.
\end{lemma}

\begin{proof} Let $\overline{G(p)}$ be first-order $\X$-rigid and let $\omega^{(n)}$ be a proper $(\X',\X)$-asymptotic equilibrium stress for $G(p)$. For each $n$ let $\mu^{(n)}$ be the nonnegative stress field with $\mu^{(n)}_e =|\omega^{(n)}_e|$. Let $u\in \X\otimes \bR^d$ be an infinitesimal flex so that the stress field $\sigma= R(G(p))u$ is nonnegative, with $\sigma_e\geq 0$ for all $e$. 
We have
\[
\mu^{(n)}R(G(p))u=\mu^{(n)}\cdot \sigma = \sum_e \mu^{(n)}_e\sigma_e,
\] 
where the sum is finite for each $n$ and the terms $\mu^{(n)}_e\sigma_e$ are nonnegative.
On the other hand since $\mu^{(n)}R(G(p))$ tends to zero, these finite sums also tend to zero, as $n$ tends to infinity, and so $\mu^n_e \geq \delta_e$ for some $\delta_e$ and all $n\geq n_e$.
It follows that $\sigma_e=0$ for each $e$. Thus $u$ is an infinitesimal flex for $\overline{G(p)}$ and so $u$ is a rigid motion infinitesimal flex, as required.
\end{proof}

\subsection{Closed convex cones and first-order $\X$-rigidity.}

We now obtain the following generalisation of the Roth-Whiteley theorem. Note that the sufficiency direction follows from Lemma \ref{l:generalsufficiency}.

\begin{thm}\label{t:generalrothwhiteley}
Let $G(p)$ be a countable cable-strut tensegrity with uniform structure and let $\X$ be one of the Banach sequence spaces $c_0, \ell^q,$ for $ 1\leq q <\infty$. Then $G(p)$ is first-order $\X$-rigid if and only if $\overline{G(p)}$ is first-order $\X$-rigid and there exists a proper $\X'$-asymptotic equilibrium stress.
\end{thm}


A consequence of the Hahn-Banach theorem for locally convex topological vector spaces is the following general separation theorem.
For a proof see Theorem 3.9 in Conway \cite{con}.

\begin{thm}\label{t:conwayseparation} Let $\X$ be a real locally convex topological vector space and let $A$ and $B$ be two disjoint closed
convex subsets of $\X$. If $B$ is compact, then $A$ and $B$ are strictly separated.
\end{thm}


Strict separation means that there exists a real value $\alpha$ and a continuous linear functional $f$ such that $f(b)<\alpha< f(a)$ for all $b\in B, a \in A$.
Applying this to a singleton set $B=\{w\}$ and a closed convex cone $A$, with $w\notin A$, we have $\alpha < 0 = f(0)$. Also, since $A$ is a cone we cannot have $f(a) <0$ for any $a \in A$ for otherwise $f(\beta a) < \alpha $ for some $\beta>0$. Thus we have the following separation lemma generalising Lemma \ref{l:separationHilbert}.

\begin{lemma}\label{l:coneseparation} Let $C$ be a closed convex cone in a real locally convex topological vector space  $\X$ and let $w\in \X\backslash C$. Then there is a continuous linear functional $f$ on $\X$ with $f(w)< 0 \leq f(x)$ for all $x\in C$.
\end{lemma}


 Henceforth we assume that $\X$ is one of the Banach sequence spaces  $\ell^q, 1\leq q<\infty$, or $c_0$, and we write $\hat{x}$ for the image of $x \in \X$ under the canonical injection from $\X$ to its second dual space $\X''$.

As a substitute for the double dual cone identity $\hat{C}=C^{**}$ in finite dimensions we have the following lemma. 

\begin{lemma}\label{l:doubledualconeSUB}
Let $C$ be a closed convex cone in a Banach space $\Y$ and let $\hat{\Y}$ be the natural embedding of $\Y$ in the second dual $\Y''$. Then $\hat{C}= C^{**}\cap \hat{\Y}$.
\end{lemma}

\begin{proof}
The inclusion of $\hat{C}$ in the intersection is clear. Let $w\in \X\backslash  C$. By Lemma \ref{l:coneseparation} there exists $f$ in $C^*$ with $f(w)< 0$. In particular $\hat{w} \notin C^{**}$.
\end{proof}

Example \ref{e:RWfails} shows that the dichotomy lemma does not extend in a literal way to infinite matrices. However we have the following extension to tensegrities with uniform structure and this completes the proof of Theorem \ref{t:generalrothwhiteley}.  We write $\bR^E_+$ to denote the nonnegative cone in the direct product space $\bR^E=\prod_{e\in E}\bR$.

\begin{lemma}\label{l:dichotomyBanach}
Let $G(p)$ be a countable tensegrity in $\bR^d$ with uniform structure which is first-order $\X$-rigid. Then precisely one of the following is true.
\medskip

(i) There exists $u$ in $\X \otimes \bR^d$ with $R(G(p))u$ a nonzero element of the cone  $\bR_+^E$.

(ii) There exists a proper $\X'$-asymptotic equilibrium stress $\omega^{(n)}$ for $G(p)$.
\end{lemma}

\begin{proof}
If (i) holds then (ii) is false by the sufficiency direction of Theorem \ref{t:generalrothwhiteley}. 

Assume then that (i) does not hold. Since $G(p)$ has uniform structure  $R(G(p))u$ lies in $\X\otimes \bR$ for any velocity field $u$ in $\X \otimes \bR^d$.
 Let $X=\{x_1, x_2, \dots \}$ be the rows of $R(G(p))$ and let $C(X)$ be the closed convex cone in $\X'\otimes \bR^d$ generated by $X$. As in the earlier dichotomy lemma, any velocity field $u$ may be viewed as a continuous linear functional on the cone $C(X)$. The failure of (i) is the assertion that
$C(X)^* = X^\perp$ where $X^\perp$ denotes the annihilator of $X$ in $(\X'\otimes \bR^d)'$. In particular the dual cone $C(X)^*$ is a linear subspace of $(\X'\otimes \bR^d)'$ and so by, Lemma \ref{l:doubledualconeSUB}, $C(X)$ is also a subspace. In particular
if $y$ is the element $x_1+\frac{1}{2}x_2 + \frac{1}{4}x_2 + \dots $ in $C(X)$
then $-y$ lies in $ C(X)$ and so $-y$ is a limit in $\X'\otimes\bR^d$ of a sequence of finite sums $\sigma^{(n)}=\sum_i \lambda_i^nx_i$, for $n=1,2,\dots $, where the coefficients are nonnegative. It follows that the sequence
\[
 x_1+\frac{1}{2}x_2 + \dots + \frac{1}{2^n}x_n + \sigma^{(n)}, n=1, 2, \dots ,
\]
tends to $0$ in $\X'\otimes \bR^d$. Also, this sequence has the form
\[
\mu^{(n)}R(G(p)), n=1, 2, \dots ,
\] 
where for each $k$ we have $\mu^{(n)}_k \geq \frac{1}{2^k}$ for $n\geq k$.  Let $\omega_k^{(n)}=\mu_k^{(n)}$ (resp. $-\mu_k^{(n)}$) if $x_k$ is the row for a strut (resp. cable). Then $\omega^{(n)}$ is the desired proper $\X'$-asymptotic equilibrium stress.
\end{proof}

We remark that the relationship between $\X$-infinitesimal rigidity and appropriate forms of continuous rigidity has not been greatly explored. For example, while the kagome periodic framework in $\bR^2$ has  bounded infinitesimal flexes it seems to be  an open problem whether it has a smooth motion (or even a continuous motion) $p(t)$ for which the maximum deviations of the joints, that is, the quantities $\delta_k(p)= \sup_t |p_k(t)-p(0)|$,  are uniformly bounded. 
See also Section 2 of Owen and Power \cite{owe-pow-crystal}. We note also that our  remarks in Section \ref{ss:examples} suggest that $\X$-infinitesimal rigidity is at best a sufficient condition (for trivial reasons) for a directed form of continuous rigidity with respect to $\X$.

\section{Prestress stability, second order rigidity and continuous rigidity}
We first consider finite bar-joint frameworks and give direct proofs that prestress stability implies continuous rigidity and that second order rigidity is an intermediate property.

\subsection{PS, 2OR and CR for finite frameworks}\label{ss:LRfinite}
A classical result of Asimow and Roth \cite{asi-rot}, \cite{asi-rot-II} in the rigidity theory of {finite} bar-joint frameworks $(G,p)$ is that
if $(G,p)$ is infinitesimally rigid (IR) then it is continuously rigid (CR). The proof uses the implicit function theorem to establish this implication under the assumption that $p$ is a regular point in $\bR^{d|V|}$ for the distance function $f_G$ from $\bR^{d|V|}$ to $\bR^{|E|}$. See Asimow and Roth \cite{asi-rot} or Roth \cite{rot}. On the other hand with the maximal rank characterisation of infinitesimal rigidity in \cite{asi-rot} it is shown, in \cite{asi-rot-II}, that if $(G,p)$ is infinitesimally rigid then $p$ is necessarily a regular point.

It is also true that infinitesimal rigidity implies continuous rigidity for finite tensegrities, a fact due to Connelly - see Roth-Whiteley \cite{rot-whi}, Theorem 5.7. 
 
For a finite \emph{generic} framework, being continuously rigid is equivalent to being first-order rigid  \cite{asi-rot}. The nongeneric bar-joint frameworks of Figure \ref{f:LRbutnotIR} show that this equivalence fails in general.
That the first framework is CR and not IR is clear, since the 3 lower joints are colinear. It is less evident that the second framework is continuously rigid.  
One can give a direct ad hoc proof but better insights come from the consideration of equilibrium stresses and prestress stability.

\begin{center}
\begin{figure}[ht]
\centering
\includegraphics[width=3.5cm]{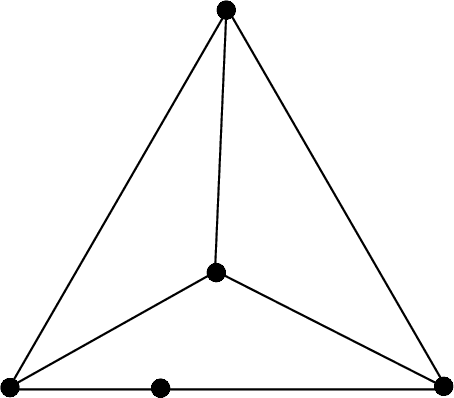}\quad \quad \quad \quad
\includegraphics[width=3.5cm]{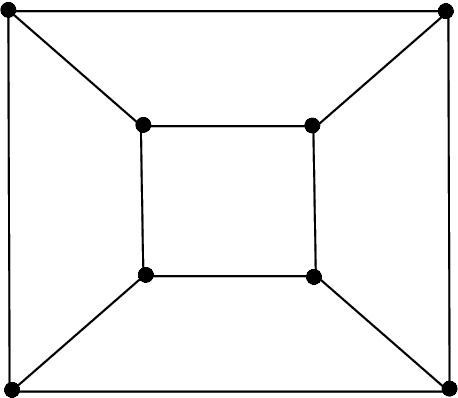}
 \caption{CR but not IR.}
 \label{f:LRbutnotIR}
\end{figure}
\end{center} 

A general \emph{stress} of a bar-joint framework $(G,p)$ is any real field $\omega : E \to \bR$ on the edges and an \emph{equilibrium stress} is a stress with $\omega R(G,p)=0$. 
This is equivalent to an equilibrium condition at each vertex, namely
$\sum_{j}\omega_{ij} (p_i-p_j) =0,$ 
for each vertex  $v_i$.

In the next definition we adopt the convention that $\omega_{ij}=0$ if $i=j$ or if $ij$ is not an edge. Also we have $\omega_{ij}=\omega_{ji}$ and a summation over pairs $i,j$ with $i<j$ indicates a summation over all the edges of $G$. We say that an infinitesimal flex is \emph{trivial} if it is a rigid motion infinitesimal flex.

\begin{definition}\label{d:PS}
A finite bar-joint framework $(G,p)$ is \emph{prestress stable} (PS) if there is an equilibrium stress $\omega$ such that for each nontrivial infinitesimal flex $u$,
\[
\sum_{i<j} \omega_{ij}|u_i-u_j|^2 >0.
\]
\end{definition}

It follows immediately from the definitions that first-order rigidity implies prestress stability. To see that prestress stability implies continuous rigidity we make use of the stress energy functions introduced by Connelly \cite{con-energy}.

For a general stress  $\omega$ for $(G,p)$  the \emph{stress energy function} $E_\omega : \bR^{nd} \to \bR$ is given by
\[
E_\omega(q) = \sum_{i<j} \omega_{ij}|q_i-q_j|^2.
\]

In the presence of prestress stability, for an equilibrium stress $\omega$, the next lemma shows that the second derivative at $t=0$ of $E_\omega(p+tu)$ is positive for a nontrivial infinitesimal flex $u$. The usefulness of the lemma in the proof of Theorem \ref{t:PSimpliesLR} comes from the stability fact that
if $u'$ is close enough to $u$ then the second derivative at $t=0$ of $E_\omega(p+tu')$ is also positive.



\begin{lemma}\label{l:secondderivative}
If $\omega$ is any stress for $(G,p)$ and 
$q, u$ are vectors in $\bR^{dn}$
then 
\[
\frac{d^2}{dt^2}E_\omega (q+tu)|_{t=0}= 2\sum_{i<j} \omega_{ij}|u_i-u_j|^2.
\]
\end{lemma}

\noindent 
\begin{proof} Expanding the terms  
\[
\omega_{ij}|(q_i+tu_i-(q_j+tu_j)|^2 = \omega_{ij}\langle q_i+tu_i-(q_j+tu_j),q_i+tu_i-(q_j+tu_j)\rangle
\]  
 we have
\[
E_\omega (q+tu) = \sum_{i<j} \omega_{ij}\big(|q_i-q_j|^2
+2t\langle q_i-q_j,u_i-u_j\rangle +t^2|u_i-u_j|^2\big)
\]
and the desired identity follows.
\end{proof}

Let us also note that the derivative of $E_\omega (p+tu)$ at $t=0$ is given by
\[
\frac{1}{2}\frac{d}{dt}E_\omega (p+tu)|_{t=0}= \sum_{i<j} \omega_{ij}\langle p_i-p_j,u_i-u_j\rangle = \sum_{i<j} \omega_{ij}\langle p_i-p_j,u_i\rangle + \sum_{i<j} \omega_{ij}\langle p_j-p_i,u_j\rangle
\]
\[
=\sum_{i<j} \omega_{ij}\langle p_i-p_j,u_i\rangle + \sum_{i<j} \omega_{ij}\langle p_j-p_i,u_j\rangle = \sum_{i<j} \omega_{ij}\langle p_i-p_j,u_i\rangle + \sum_{j<i} \omega_{ji}\langle p_i-p_j,u_i\rangle,
\]
which is zero when $\omega$ is an equilibrium stress.

\begin{thm}\label{t:PSimpliesLR}
Prestress stability implies continuous rigidity for finite bar-joint frameworks.
\end{thm}

\begin{proof}
Let $(G,p)$ be prestress stable with equilibrium stress $\omega$ as in Definition \ref{d:PS}. Assume that $(G,p)$ is not continuously rigid. Then there is a continuously differentiable motion $p(t)$ of $(G,p)$ with $p'(0)$ a nontrivial infinitesimal flex $u$. 
The  real-valued function $t \to E_\omega(p+tu)$ has vanishing derivative at $t=0$, since $\omega$ is an equilibrium stress, and so by Lemma  \ref{l:secondderivative} it 
has the form $A(u)t^2+C$ where $C$ is constant and
$A(u)=2E_\omega(u)=2\sum_{i<j} \omega_{ij}|u_i-u_j|^2 >0$. 
In particular $E_\omega(p+tu)$ is strictly increasing on $\bR_+$.

Consider the finite cone $C(u,\delta)$ of velocity fields in $\bR^{nd}$ given by 
\[
C(u,\delta)=
\{ su': s\in[0,1], |u'-u|\leq \delta\}
\] where $\delta>0$ is such that $A(u') >0$ when $ |u'-u|\leq \delta$. As in the previous paragraph $E_\omega(p+tu')$ is strictly increasing on $\bR_+$ from which it follows that $E_\omega (q)> E_\omega (p)$ for all $q\neq p$ in $p+C(u,\delta)$. 
 
For any continuously differentiable path $s \to q(s)$ in $\bR^{nd}$ with $q(0)=p$ and $q'(0)=u$ the point $q(s)$ is equal to to $p+su$ to first order, and so $q(s)$ lies in $p+C(u,\delta)$ for some sufficiently small $s>0$.
Considering this for the motion $p(s)$ gives the desired contradiction since in this case the bar lengths are preserved and the function  $s\to E_\omega(p(s))$ is constant.
\end{proof}


\subsection{Second order rigidity}

In view of the entries of the rigidity matrix $R(G,u)$ we see that 
\[
\sum_{i<j} \omega_{ij}|u_i-u_j|^2 = 
\sum_{i<j} \omega_{ij}\langle u_i-u_j, u_i-u_j\rangle = \omega R(G,u)u.
\] 
Thus, prestress stability may be reworded as the  following property.
\begin{center}
\emph{There exists  $\omega$  with  $\omega R(G,p) =0$ and  $\omega R(G,u)u >0$ for all nontrivial flexes $u$.} 
\end{center}

This compact reformulation is helpful when making comparisons with another form of rigidity, namely \emph{second order rigidity} (2OR). 
The idea of second order rigidity is that a nontrivial infinitesimal flex $u$ of $(G,p)$ may exist but for each such $u$ there is no differentiable continuous motion $p(t)$ starting in the direction $u$, that is, with $p'(0)=u$. One way to effect this obstruction, in view of the proof of Theorem \ref{t:PSimpliesLR}, is the presence  an equilibrium stress $\omega$ specifically for $u$ in the sense that $ \sum_{i<j} \omega_{ij}|u_i-u_j|^2 >0.$
This suggests the following additional definition which is evidently a weakening of prestress stability. 

\begin{definition}\label{d:secondorderrigid}
A finite bar-joint framework $(G,p)$ is \emph{weakly prestress stable} (WPS) if
 for each nontrivial infinitesimal flex  $u$  there is  an equilibrium stress $\omega$ with $ \sum_{i<j} \omega_{ij}|u_i-u_j|^2 >0.$ 
\end{definition}

In fact Connelly and Whiteley \cite{con-whi} have shown that this property is equivalent to second order rigidity as defined in terms of (the nonexistence of) nontrivial second order flexes. Such a flex for $(G,p)$ is a pair $(u,a)$ where $u, a$ are vectors in $\bR^{nd}$, representing a nontrivial velocity flex $u$ and an acceleration vector $a$. This pair should satisfy the equations arising from the second derivative of the constraint equations for a smooth motion $p(t)$, and they have the form
\[
R(G,p)u=0,  \quad R(G,p)a +R(G,u)u =0.
\]

Since the image space $R(G,p)\bR^{nd}$ is orthogonal to $\ker R(G,p)^T$, the nullspace of the transpose, it follows, from the Fredholm alternative, that either there is a solution $a$ to the equation $R(G,p)a=-R(G,u)u$ 
or there exists $\omega$ in $\ker R(G,p)^T$ with $\omega \cdot (R(G,u)u)\neq 0$. This shows that second order rigidity implies weakly prestress stable. 
In view of our comments preceding Definition \ref{d:secondorderrigid} the reverse implication holds and the properties are equivalent. In particular this gives an alternative to Connelly's original proof \cite{con-1980} (see also \cite{con-gue}) that second order rigidity implies continuous rigidity.

To summarise, for finite bar-joint frameworks we have
\[
IR \implies PS \implies 2OR (=WPS) \implies CR
\]

\begin{example}\label{e:CRbutnotIR_Labelled}
 Let us revisit the square within square framework of Figure \ref{f:LRbutnotIR_infinite}.  The outer joints are $p_1=(-1,1), p_2=(1,1), p_3= (1,-1), p_4=(-1,-1)$, the inner joints are
$p_5=(-1/2,1/2), p_6=(1/2,1/2), p_7= (1/2,-1/2), p_8=(-1/2,-1/2)$, and an equilibrium stress is indicated in Figure \ref{f:LRbutnotIR_Labelled}.
\begin{center}
\begin{figure}[ht]
\centering
\includegraphics[width=5cm]{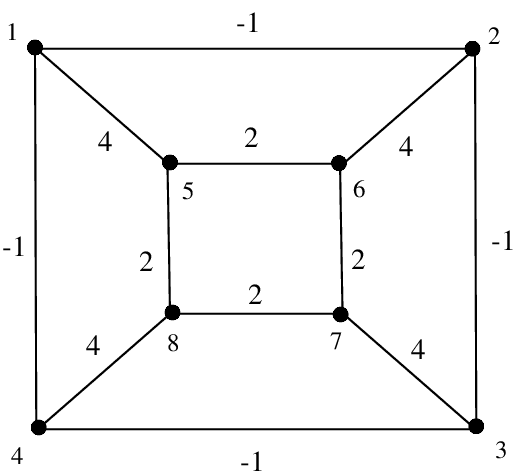}
 \caption{An equilibrium stress.}
 \label{f:LRbutnotIR_Labelled}
\end{figure}
\end{center} 
The nontrivial ``parallel motion" infinitesimal flex $z$ is defined by $z_1=z_2=z_5=z_6=(1,0)$, with $z_k=(0,0)$ for other values of $k$, and the infinitesimal flex $w$ is infinitesimal rotation of the inner square, with $w_5=(1,1), w_6=(1,-1), w_7=(-1,-1), w_8=(-1,1)$ and $w_k=(0,0)$ for other values of $k$. Up to a scalar factor, and a trivial infinitesimal flex, a typical nontrivial infinitesimal flex has the form $\alpha z+\beta w$ with $0\leq \alpha, \beta \leq1$ not both zero. We have 
\[
E_\omega(\alpha z+\beta w)=
\sum_{ij}\omega_{ij}\left[\alpha^2|z_i-z_j|^2+\beta^2|w_i-w_j|^2-2\alpha \beta \langle (z_i-z_j), (w_i-w_j)\rangle\right].
\]
We note that with the exception of $i,j=5,8$ or $6,7$ at least one of the pairs of vectors $z_i-z_j, w_i-w_j$ is zero. Moreover, for these exceptional values the pair of vectors is an orthogonal pair, and so there is no contribution to the sum from the inner product terms. It remains to note that despite the presence of negative stresses $\omega_{ij}$ (of value -1) on the outer edges we have
$\sum_{ij}\omega_{ij} \alpha^2|z_i-z_j|^2 >0$ if $\alpha>0$. It follows that $E_\omega(\alpha z+\beta w)$ is positive and so the framework is prestress stable and hence continuously rigid.
\end{example}

\subsection{Prestress stability and stress matrices} Prestress stability can also be defined in terms the \emph{stress matrix} of an equilibrium stress, as we now demonstrate.

For any velocity vector $u$ and stress field $\omega$ associated with the edges $ij$ of a finite bar-joint framework $G$ (and with $\omega_{ij}=0$ for a nonedge) we have 
\[
\omega R(G,u)u= \sum_{i<j} \omega_{ij}|u_i-u_j|^2 = 
\sum_{i<j} \omega_{ij}\langle u_i-u_j, u_i-u_j\rangle ,
\]
which in turn is equal to
\[
\sum_{i<j} \omega_{ij}(\langle u_i,u_i \rangle -2 \langle u_i,u_j\rangle  +\langle u_j,u_j \rangle ) = 
 \sum_{j:j\neq i} \omega_{ij}|u_i^2|  -\sum_{i<j} 2 \omega_{ij}\langle u_i,u_j\rangle . 
\] 
On the other hand a general square symmetric $n \times n$ matrix $\Omega_{ij}$, with entries $a_{ij}$, has a quadratic form
\[
x^T\Omega x = \sum_i x_i (\Omega x)_i 
= \sum_{i,j} a_{ij}x_ix_j = \sum_i a_{ii}x_i^2 + 2\sum_{i<j} a_{ij}x_ix_j,
\]
and $\Omega\otimes I_d$ similarly has the quadratic form
\[
u^T(\Omega \otimes I_d) u 
= \sum_{i,j} a_{ij}\langle u_i,u_j\rangle = \sum_i a_{ii}|u_i|^2 + 2\sum_{i<j} a_{ij}\langle u_i, u_j\rangle .
\]
So we see that 
\[\omega R(G,u)u=u^T(\Omega \otimes I_d) u, \quad \mbox{where}\quad
\Omega_{ij}=-\omega_{ij}, \quad \mbox{for}\quad  i\neq j, \quad \mbox{and}\quad  \Omega_{ii}= \sum_{j:j\neq i}\omega_{ij}.
\]
In particular the row sums and the column sums of $\Omega$ are zero. The matrix $\Omega$ is defined to be the \emph{stress matrix} for $\omega$ and it follows that we have the following alternative definition of prestress stability.
\medskip

\begin{prop}\label{p:stressmatrixPS} A finite bar-joint framework 
$(G,p)$  in $\bR^d$ is prestress stable if and only if there exists an equilibrium stress matrix $\Omega$ for $p$  with 
$u^T(\Omega \otimes I_d) u >0$ for each nontrivial infinitesimal flex $u$. 
\end{prop}

Prestress stability can also be defined for finite tensegrities in the same way to that of Definition \ref{d:PS}, but in terms of the stress energy function associated with a {proper} equilibrium flex. See Connelly \cite{con-1980},  Connelly and Whiteley \cite{con-whi} and Connelly and Guest \cite{con-gue}, for example. Here one can also find examples of tensegrities, such as various Cauchy polygons, which are continuously rigid by virtue of having stress matrices which are positive semidefinite on a complementary subspace to the space of rigid motion infinitesimal flexes.

\subsection{Prestress stability and rigidity for infinite frameworks}\label{ss:LRinfinite}
We first note an elementary extension of Theorem \ref{t:PSimpliesLR} to countable bar-joint frameworks.

\begin{definition}\label{d:sequentiallyPS} 
A countable framework $(G,p)$ is \emph{sequentially prestress stable} if there is an increasing sequence of prestress stable subframeworks $(G_n,p)$, determined by a chain $G_1 \subset G_2 \subset ... $ of finite subgraphs whose union is $G$. 
\end{definition}

\begin{definition}\label{d:sequentiallyctsrigid}
A countable framework $(G,p)$ is \emph{sequentially continuously rigid} if there is an increasing sequence of continuously rigid subframeworks $(G_n,p)$, determined by a chain $G_1 \subset G_2 \subset ... $ of finite subgraphs whose union is $G$. 
\end{definition}

\begin{thm}\label{t:SPSimpliesCR}
If $(G,p)$ is sequentially prestress stable then it is sequentially continuously rigid and hence continuously rigid.
\end{thm}

This theorem follows readily from Theorem \ref{t:PSimpliesLR}.

\begin{example} Figure \ref{f:LRbutnotIR_infinite} indicates an infinite bar-joint framework $(G,p)$. The joint positions are located dyadically; the outer joints are $p_1=(-1,-1), p_2=(-1,1), p_3= (1,1), p_4=(1,-1)$ and the other joints have the form $\frac{1}{2^k}v_j, j=1,2,3,4, k= 1,2,\dots$. 
The framework $(G,p)$ has an infinite dimensional flex space, since the quadruple of joints of the square subframeworks with bar lengths $(1/2)^k$, for $k=1,2,\dots $, is the support of an infinitesimal rotation flex, say $w^{(k)}$, for $k=1,2,\dots $. We assume, as before, that $|w^{(k)}|=\sqrt{2}$ for each $k$.
It follows from the continuous rigidity of the double square framework of Example \ref{f:LRbutnotIR} and a simple induction argument that the framework is sequentially continuously rigid. 
 Also it can be shown that it is sequentially prestress stable.
\begin{center}
\begin{figure}[ht]
\centering
\includegraphics[width=7cm]{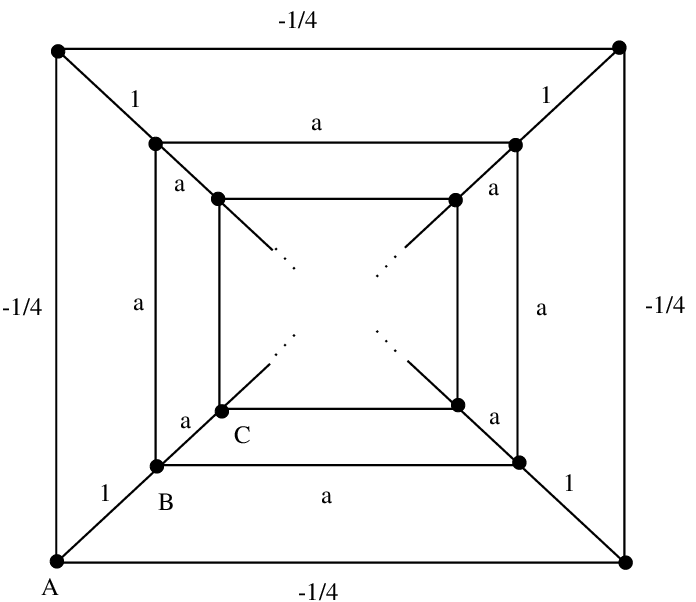}
 \caption{A partially defined equilibrium stress.}
 \label{f:LRbutnotIR_infinite}
\end{figure}
\end{center} 
We also observe that it is possible to define an equilibrium flex $\omega$ for $(G,p)$ which is summable in the sense that the sequence of weights $\omega_e$ gives an absolutely convergent series.
The equilibrium condition is satisfied for the vertex $v_1$ at $A$.
Also the equilibrium condition is satisfied at $B$ if $ a\frac{1}{8} + a\frac{1}{2}=1/4$, 
that is if $a=\frac{4}{5}$. 
It follows that there is a rotationally symmetric summable equilibrium stress $\omega$  with stress values $(\frac{4}{5})^k$ for $k=2,3,\dots$. 
\end{example}

For countable frameworks infinitesimal rigidity does not ensure continuous rigidity, as we observe below in Section \ref{ss:examples} and in Remark \ref{r:CRandDCR}. In view of this and the lack of algebraic geometry methods, it is natural to consider some variants of continuous rigidity in this setting.

Define the space of bounded displacement vectors (or bounded velocity vectors) $u=(u_k)_{k\in \bN}$  as the normed space $\ell^\infty \otimes \bR^d$ with norm given by the supremum, so that $\|u\|=\sup_k |u_k |$, with $|\cdot|$ the usual Euclidean norm on $\bR^d$. We say that a (continuous) motion $p(t), 0\leq t\leq1,$ of $(G,p)$ is \emph{bounded motion} if  $p(t)-p(0)$ belongs to  $\ell^\infty \otimes \bR^d$ for all $t$.

\begin{definition}\label{d:boundedlysmooth}
A \emph{directed} bounded motion of 
$(G,p)$ is a continuous motion $p(t), t\in[0,1],$ such that 
the limit 
\[
\lim_{t\to 0+} (p(t)-p(0))/t = (p_k'(0))_{k\in \bN}
\]
exists with respect to convergence in the displacement space $\ell^\infty \otimes \bR^d$. 
Moreover, the directed bounded motion $p(t)$ is \emph{proper} if  
the infinitesimal flex $p'(0)$ is nonzero, and is \emph{nontrivial} if $p'(0)$ is nontrivial. 
\end{definition}

\begin{definition}\label{d:smoothlocalrigidity}
A countable bar-joint framework in $\bR^d$ is said to be \emph{directedly boundedly continuously rigid}, or, for brevity, \emph{directedly continuously rigid} (DCR) if it has no nontrivial directed bounded motions.
\end{definition}

For finite bar-joint frameworks we may similarly define  \emph{directed motions} and \emph{directed continuous rigidity}. That IR implies DCR is evident from the definitions. On the other hand the equivalence of CR and DCR for finite frameworks depends on algebraic geometry arguments, as indicated in Remark \ref{r:CRandDCR}.

Let us say that a countable framework  $(G,p)$ is \emph{boundedly infinitesimally rigid} (BIR) if there exists no infinitesimal flexes which are nontrivial and bounded. From the definitions we have that that IR implies BIR, and BIR  implies DCR. 

We next turn to a generalisation of the finite framework result that PS  (a weakening of IR) implies DCR. This is done in Theorem \ref{t:BPSimpliesSLR} for frameworks with bounded bar lengths. 

\begin{definition}\label{d:BPS} A  bar-joint framework $(G,p)$ with bounded bar lengths is \emph{boundedly prestress stable} (BPS) if there is a \emph{summable} equilibium stress $\omega$, in the sense that $\sum_{i<j}|\omega_{ij}|$ is finite, and  
for all bounded nontrivial infinitesimal flexes $u$ of $(G,p)$,
\[
\sum_{i<j} \omega_{ij}|u_i-u_j|^2 >0.
\]
\end{definition}

We note that when the bar lengths are bounded  $(G,p)$ has \emph{finite stress energy for $\omega$} in the sense that
\[
E_\omega(p) = \sum_{i<j} \omega_{ij}|p_i-p_j|^2 <\infty.
\]
Once again we refer to the function $q \to E_\omega(q)$ as the \emph{stress energy function} for $\omega$ and $G$.
 

\begin{example}
The framework $(G,p)$ of Example \ref{f:LRbutnotIR_infinite} is boundedly prestress stable for the summable equilibrium stress $\omega$ given there. To see this let $u$ be a nontrivial, bounded infinitesimal flex. 
We may assume that $u$ has zero velocities at $p_3, p_4$. In this case $u$  admits an infinite sum representation 
\[
u= \alpha z + \sum_k\beta_k w^{(k)},
\] 
where $(\beta_k)$ is a bounded sequence and where $z$ is the ``parallel motion" infinitesimal flex whose velocities are zero on the lower joints of the square subframeworks and equal to $(1,0)$ on the upper joints. The positivity and finiteness of $E_\omega(u)$ follow as in Example \ref{e:CRbutnotIR_Labelled} since the resulting series are absolutely convergent.
\end{example}

\begin{example}
Let $(G,p)$ be the lacunary linear framework in $\bR^2$ with bars corresponding to the intervals 
$
\dots , [-7,-3],[-3,-1],[-1,0],
[0,1],[1,3],[3,7], \dots
.$ 
\begin{center}
\begin{figure}[ht]
\centering
\includegraphics[width=8cm]{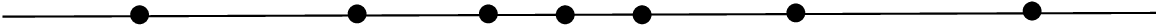}
 \caption{A lacunary linear framework.}
 \label{f:dyadiclinear}
\end{figure}
\end{center} 
Evidently $(G,p)$ is not BSR. Let $\omega$ be the summable equilibrium stress with $\omega_e = \frac{1}{2^k}$ if the bar for edge $e$ has length $2^k$. Every equilibrium stress is in fact a scalar multiple of $\omega$. Note that for any uniformly bounded infinitesimal flex $u$ (or indeed any bounded velocity field) the quantity $E_\omega(u)$ is finite. On the other hand $E_\omega(p)$ is not finite and  $(G,p)$ has unbounded bar lengths.
\end{example}


\begin{lemma}\label{l:summablestressse2ndderiv}
Let $\omega$ be a summable equilibrium stress for a bar-joint framework $(G,q)$ with bounded bar lengths, and 
let $u$ be a bounded velocity vector.
Then 
\[
\frac{d^2}{dt^2}E_\omega (q+tu)|_{t=0}= 2\sum_{i<j} \omega_{ij}|u_i-u_j|^2.
\]
\end{lemma}

\begin{proof}
We have
\[
\omega_{ij}|(q_i+tu_i-(q_j+tu_j)|^2 = \omega_{ij}\langle q_i+tu_i-(q_j+tu_j),q_i+tu_i-(q_j+tu_j)\rangle,
\]  
and so 
\[
E_\omega (q+tu) = \sum_{i<j} \omega_{ij}|q_i-q_j|^2
+ \sum_{i<j}\omega_{ij} 2t\langle q_i-q_j,u_i-u_j\rangle +\sum_{i<j}\omega_{ij}t^2|u_i-u_j|^2.
\]
Each of these series is absolutely convergent and so the desired identity follows.
\end{proof}

From the expansion of $E_\omega (q+tu)$ above it follows that when $\omega$ is summable and $(G,p)$ has bounded bar lengths 
then for a bounded infinitesimal flex $u$ we have
\[
\frac{d}{dt}E_\omega (p+tu)|_{t=0}= 2\sum_{i<j} \omega_{ij}\langle p_i-p_j,u_i-u_j\rangle,
\]
where the series is absolutely convergent.
In particular, when $\omega$ is a summable equilibrium stress for $(G,p)$ then,  as with finite frameworks, this derivative is zero.

\begin{thm}\label{t:BPSimpliesSLR}
Let $G$ be a countably infinite bar-joint framework with bounded bar lengths. If $(G,p)$ is boundedly prestress stable then it is directedly continuously rigid. 
\end{thm}

\begin{proof} Assume that $(G,p)$ is BPS, with associated summable equilibrium stress $\omega$, and that $(G,p)$ is not DCR. Then there is 
a boundedly smooth motion $p(t)$ of $(G,p)$ with $p'(0)$ equal to a nontrivial bounded infinitesimal flex $u=(u_k)$.
As in the proof of Theorem \ref{t:PSimpliesLR} we write $A(u)$ for
$2\sum_{i<j} \omega_{ij}|u_i-u_j|^2$. Note that the function $u' \to A(u')$
is continuous in the sense that for all $\epsilon>0$ there exists $\delta>0$
such that $|A(u')-A(u)|<\delta$ for $\|u'-u\|<\delta$. To see this note that
\begin{equation*}
\begin{split}
A(u)-A(u') &=\sum_{i<j}\omega_{ij}\big( |u_i-u_j|^2-|u_i'-u_j'|^2\big)\\
&\leq \sum_{i<j}\omega_{ij}\big( (|u'_i-u'_j|+2\delta)^2-|u_i'-u_j'|^2\big)\\
&= \sum_{i<j}\omega_{ij}\big(4\delta|u'_i-u'_j|+4\delta^2  \big)
\end{split}
\end{equation*}
Since $u'$ is bounded and $\omega$ is summable this is less than $\epsilon$ for suitably small $\delta$. The same can be said for $A(u')-A(u)$.

As in the proof of Theorem \ref{t:PSimpliesLR} it now follows from this continuity that for some $\delta>0$ the function $q \to E_\omega(q)$ is increasing on each ray of the cone $p+C(u,\delta)$.
However, for the motion $p(t)$ the strong differentiability condition for $t=0$ shows that $\|p(s) - (p+su)\|<\delta s$, for some small $s>0$.
Thus $p(s)$ lies in the cone $p+C(u,\delta)$ and so $E_\omega(p(0)) < E_\omega(p(s))$, which is the desired contradiction.
 \end{proof}



It is also straightforward to formulate an associated form of weak prestress stability.

\begin{definition}\label{d:weaklBPS} A countable bar-joint framework is \emph{weakly boundedly prestress stable} (WBPS) if
 for each nontrivial bounded infinitesimal flex  $u$  there is a summable  equilibrium stress $\omega$ with $ \sum_{i<j} \omega_{ij}|u_i-u_j|^2 >0.$ 
\end{definition}



In parallel with finite frameworks the proof of Theorem \ref{t:BPSimpliesSLR} shows that WBPS implies DCR and so we have the following hierarchy.

\begin{thm}\label{t:IRtoBSRchain} The following implications hold for countable bar-joint frameworks with bounded bar lengths.
\[
IR \implies BIR  \implies BPS \implies WBPS \implies DCR
\]
\end{thm}
\medskip

\subsection{Examples and remarks}\label{ss:examples}
Figure \ref{f:strip_I} represents an infinite strip bar-joint framework $(G,p)$ in $\bR^2$ where the bold semiinfinite line substitutes for a rigid base and where the joints nearest the base each have two incident bars that are colinear. The framework is infinitesimally rigid due to this colinearity. On the other hand the finite framework shown in Figure \ref{f:strip_I} has a nontrivial infinitesimal flex which is nonzero at a single joint. For similar reasons $(G,p)$ is not sequentially infinitesimally rigid.
The infinite framework has a nontrivial continuous bounded motion $p(t)$  parametrised by the $t=\sin A$, for $0\leq t\leq 1/2$, that  fixes the joints of the base (not shown). See Kastis and Power \cite{kas-pow} for further details. Also $p(t)$ is differentiable for $t>0$, when the angle $A$ is acute, and nondifferentable for $t=0$. In contrast the reparametrised continuous motion with $q(t)=p(e^{-\frac{1}{t^2}})$, for $t>0$, is a nonanalytic differentiable motion of the framework with $q'(0)=0$. On the other hand since the framework is BIR it is DCR. 

The frameworks of Figure \ref{f:strip_II} and \ref{f:strip_III} no longer have the colinear property. The first has a bounded differentiable motion $p(t)$ with  $p'(0)$ a nontrivial bounded infinitesimal flex and so fails to be BIR. The second is BIR as it has no nontrivial bounded infinitesimal flex. For further examples of this kind see Owen and Power \cite{owe-pow-crystal}.

\begin{center}
\begin{figure}[ht]
\centering
\includegraphics[width=8cm]{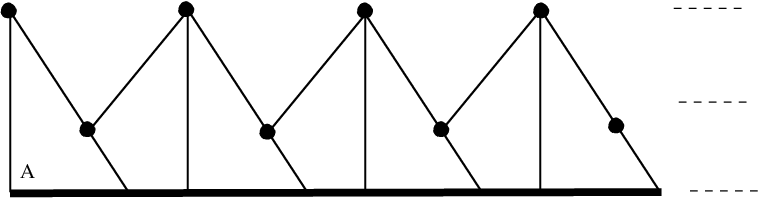}
 \caption{Infinitesimally rigid but not continuously rigid.}
 \label{f:strip_I}
\end{figure}
\end{center} 
\begin{center}
\begin{figure}[ht]
\centering
\includegraphics[width=8cm]{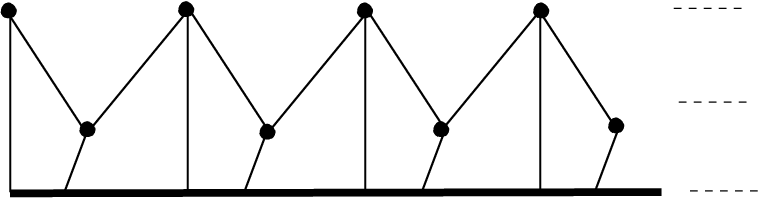}
 \caption{A nontrivial directed bounded motion $p(t)$ exists.}
 \label{f:strip_II}
\end{figure}
\end{center} 
\begin{center}
\begin{figure}[ht]
\centering
\includegraphics[width=8cm]{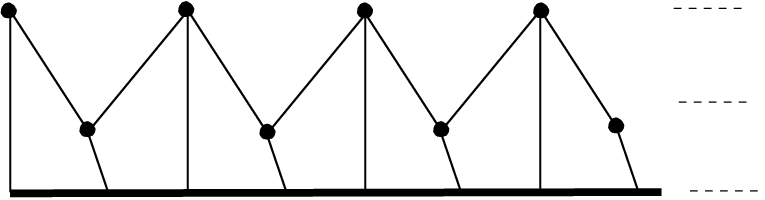}
 \caption{There is no nontrivial infinitesimal flex that is bounded.}
 \label{f:strip_III}
\end{figure}
\end{center} 
\begin{rem}\label{r:CRandDCR}
For finite frameworks continuous rigidity is in fact implied by directed continuous rigidity (DCR). We sketch the argument for this for frameworks in $\bR^2$.
The framework $(G,p)$ has a real variety of placements $q$ in $\bR^{2|V|}$ for which $(G,q)$ is equivalent to $(G,p)$. This is the configuration space of the framework. 
Consider however the real algebraic variety of equivalent placements $q$ with the property that $q_1=p_1$ and $q_2=p_2$ where $p_1p_2$ is  a bar of $(G,p)$. If $(G,p)$ is not continuously rigid then this variety  has a connected component containing $p$ which is not a singleton set. Moreover the coordinate constraint condition implies that any nonconstant continuous curve $q(t)$ in this component with $q(0)=p$ provides a nontrivial motion of $(G,p)$. By the Curve Selection Lemma in real algebraic geometry (see the useful discussion \cite{mor-rua} for example) there exists a real analytic curve $q(t)$ of this type. Such a curve may be reparametrised to provide an analytic curve $q(s)$ with $q'(0)$ nonzero,  and $q'(0)$ is therefore a nontrivial infinitesimal flex of $(G,p)$. Thus $(G,p)$ is not DCR, and so DCR implies CR. Since IR implies DCR we note that this gives another proof of the Asimow Roth theorem.

For countably infinite frameworks the notion of directed (bounded) continuous rigidity is admittedly a rather weak notion of continuous rigidity in its association with the strong requirement of a directed motion. It would be interesting to relax this requirement in some way to see wider consequences of bounded prestress stability or some variant thereof. A potential obstacle here is the fact that there exist exotic countable bar-joint frameworks in $\bR^2$ that have a unique motion for a fixed base, up to reparametrisation, and with a joint tracing an arbitrary continuous path. The underlying reason for this, roughly speaking, is that it is possible to simulate convergent series of traced algebraic curves  within a single infinite framework \cite{owe-pow-kempe}, \cite{pow-kempe}. 
\end{rem}

\medskip
\bibliographystyle{abbrv}
\def\lfhook#1{\setbox0=\hbox{#1}{\ooalign{\hidewidth
  \lower1.5ex\hbox{'}\hidewidth\crcr\unhbox0}}}

\end{document}